\documentclass{article}
\usepackage[utf8]{inputenc}
\usepackage{cite}
\usepackage{hyperref}

\title{On Newstead's Mayer-Vietoris argument in characteristic 2}
\author{Christopher Scaduto {\;}\&\, Matthew Stoffregen}
\date{}

\usepackage[a4paper, total={6in, 8.5in}]{geometry}
\usepackage{graphicx}
\usepackage{amssymb}
\usepackage{tikz-cd}
\usepackage{titling}
\usepackage{booktabs}
\usepackage{amsthm}
\usepackage{tabularx}
\usepackage{amscd}
\usepackage{caption}
\usepackage{amsmath}
\usepackage{MnSymbol}
\usepackage{tikz}
\usepackage{dsfont}
\usepackage{lipsum}
\usepackage{mwe}

\newcolumntype{Y}{>{\centering\arraybackslash}X}

\usepackage{sectsty}

\sectionfont{\fontsize{12}{15}\selectfont}

\usetikzlibrary{shapes,snakes}
\usetikzlibrary{matrix,arrows,decorations.pathmorphing}
\usetikzlibrary{positioning}

\newcommand{\Z}{\mathbb{Z}}
\newcommand{\F}{\mathbb{F}}
\newcommand{\Q}{\mathbb{Q}}

\newtheorem{prop}{Proposition}
\newtheorem{theorem}{Theorem}
\newtheorem{lemma}{Lemma}

\newtheorem{corollary}{Corollary}

\begin{document}

\maketitle 

\begin{abstract}
Consider the moduli space of framed flat $U(2)$ connections with fixed odd determinant over a surface. Newstead combined some fundamental facts about this moduli space with the Mayer-Vietoris sequence to compute its betti numbers over any field not of characteristic two. We adapt his method in characteristic two to produce conjectural recursive formulae for the mod two betti numbers of the framed moduli space which we partially verify. We also discuss the interplay with the mod two cohomology ring structure of the unframed moduli space.
\end{abstract}

\vspace{.0cm}

\section{Introduction}

Let $\Sigma_g$ be a compact surface of genus $g$, and let $N_g$ be the moduli space of flat $SU(2)$ connections on $\Sigma_g$ having holonomy $-1$ around a single puncture $p$. If we write $A_1,B_1,\ldots,A_g,B_g$ for the usual generators of the free group $\pi_1(\Sigma_g\setminus p)$, then $N_g$ is homeomorphic to $f^{-1}_g(-1)/SU(2)$, in which

\vspace{.2cm}
\[
    f_g :SU(2)^{2g} \longrightarrow SU(2), \qquad f_g(A_1,B_1,\ldots,A_g,B_g) \; = \; \prod_{i=1}^{g} [A_i,B_i],
\]
\vspace{.2cm}

\noindent and the action of $SU(2)$, which descends to a free $SO(3)$ action, is by simultaneous conjugation of the $2g$ factors. By a classical result of Narasimhan-Seshadri, $N_g$ may be identified with the moduli space of rank two stable holomorphic bundles over a Riemann surface of genus $g$ with fixed odd determinant. The moduli space of {\emph{framed}} flat connections is given by

\vspace{.15cm}
\[
    N_g^\# \; = \; f_g^{-1}(-1)
\]
\vspace{.15cm}

\noindent and forms an $SO(3)$-principal bundle over the moduli space $N_g$.\\

The betti numbers of the moduli space $N_g$ have been computed in a variety of ways. The first way, which was originally done for any coefficient field not of characteristic 2, is due to Newstead \cite{newstead-mv}. The argument, which is quite elementary, uses a Mayer-Vietoris sequence to compute formulae for the betti numbers of the framed moduli space which are recursive in $g$, and then uses the Gysin sequence for the $SO(3)$-fibration $N_g^\#$ to obtain the betti numbers for $N_g$. Subsequently, Harder-Narasimhan \cite{hn} and Atiyah-Bott \cite{ab} gave very different and more sophisticated proofs, respectively: the first number-theoretic, and the latter using infinite-dimensional Morse theory on the Yang-Mills functional. These two methods work for higher rank moduli as well. Finally, we mention the elegant proof of Thaddeus \cite{thaddeus-morse}, which shows that $(A_i,B_i) \longmapsto \text{tr}(A_g)$ is a perfect Morse-Bott function on $N_g$, as was observed by Jeffrey-Weitsman \cite{jw}.\\

Newstead's original proof shows that the integral cohomology groups of $N_g$ and $N_g^\#$ have no torsion other than 2-torsion. In their work, Atiyah-Bott showed that the integral cohomology of $N_g$ is in fact torsion-free, which can also be seen from the proof of Thaddeus. However, the space $N_g^\#$ generally has 2-torsion, as is indicated by the fact that the $g=1$ framed moduli space, which is a bundle over the point $N_1$, is homeomorphic to $SO(3)$.\\

In this article we investigate Newstead's argument in characteristic 2 with the goal of computing the cohomology of $N_g^\#$ with $\Z/2$ coefficients. Although we cannot completely compute the betti numbers from the elementary methods used here, we provide evidence for simple recursive formulae similar to Newstead's formulae for the rational betti numbers from \cite{newstead-mv}. Specifically, we conjecture that equality holds in all the inequalities appearing in the following:

\vspace{.55cm}
\begin{theorem}\label{theorem:main}
    Write $h_r^g = \dim H^r(N_g^\#;\Z/2)$. Then we have the following:\vspace{.20cm}
\begin{align*}
	h_r^{g+1} \; & \geqslant \; h_{r-2}^{g} + 2 h_{r-3}^{g} + h_{r-4}^{g} + m_r^g - m_{r-4}^g &  (r\leqslant 3g-1)\tag*{$(\text{I})_r$} \label{eq:rec1} \\
	&&\nonumber\\
	h_r^{g+1} \; & \geqslant \; 4h_{3g}^{g} + m_{3g}^g - m_{3g-3}^g &  (3g \leqslant r \leqslant 3g+3)\tag*{$(\text{II})_r$} \label{eq:rec2}\\
	&&\nonumber\\
	h_r^{g+1} \; & \geqslant \; h_{r-2}^{g} + 2 h_{r-3}^{g} + h_{r-4}^{g} + m_{r-3}^g - m_{r+1}^g &  (r\geqslant 3g+4)  \tag*{$(\text{III})_r$} \label{eq:rec3}
\nonumber
\end{align*}

\vspace{.3cm}
\noindent in which $m_r^g$ is the coefficient of $t^r$ in the polynomial $(1+t^3)^{2g}$. Further:

\begin{enumerate}
    \item[(i)] Equality holds in {\emph{\ref{eq:rec1}}} for $r\equiv 2$ (mod 3) and $r\leqslant 3g-1$.
    \item[(ii)] Equality holds in the expression for $h_k^{g+1}-h_{k-1}^{g+1}$ obtained by assuming equality in {\emph{\ref{eq:rec1}}} for $r\in\{k,k-1\}$ where $k\equiv 1$ (mod 3) and $k\leqslant 3g-1$. Also, $h_{3g+1}^{g+1} = h^{g+1}_{3g}$.

\end{enumerate}
\end{theorem}
\vspace{.55cm}

\noindent The (in)equalities obtained are immediately doubled: Poincar\'{e} duality turns (i) and (ii), which are statements for $r\leqslant 3g+1$, into statements about $r\geqslant 3g+2$. Indeed, \ref{eq:rec1} is transformed into \ref{eq:rec3} via duality, and (II)$_{3g}$ and (II)$_{3g+1}$ into (II)$_{3g+3}$ and (II)$_{3g+2}$, respectively.\\

The conjectural recursive equations obtained from imposing equality in \ref{eq:rec1}$-$\ref{eq:rec3} are remarkably similar to Newstead's equations for the rational betti numbers of \cite[Thm. 2']{newstead-mv}: there, equality in \ref{eq:rec1} is satisfied for $r\leqslant 3g+1$, and the rest of the betti numbers follow by Poincar\'{e} duality. This small difference in recursions, however, allows the $\Z/2$ betti numbers to grow much larger than the rational ones near the middle dimension. For example, the middle two $\Q$ betti numbers are zero, while our conjecture implies that the four middle $\Z/2$ betti numbers are the same and equal to

\vspace{.2cm}
\[
    2^{2g-1} - {2g-1 \choose g}.
\]
\vspace{.2cm}

\noindent The comparison of these betti numbers is further illustrated in Figure \ref{fig:introtables}. The table for the $\Z/2$ betti numbers was computed using Proposition \ref{prop:ss} below along with computations from \cite{ss-mod2}, and confirms the conjectural recursive formulae for $g\leqslant 6$. Proposition 1 computes the Leray-Serre spectral sequence for the fibration $N_g^\# \longrightarrow N_g$ in terms of the rank of multiplication by $\alpha$ on the ring $H^\ast(N_g;\Z/2)$, where $\alpha$ is the generator of $H^2(N_g;\Z/2)$. We mention that another consequence of the conjecture is the following identity between total ranks:

\vspace{.2cm}
\begin{equation}
    \dim_{\Z/2} H^\ast(N^\#_g;\Z/2) \; = \; 2\cdot \dim_{\Q} H^\ast(N^\#_g;\Q),\label{eq:double}
\end{equation}
\vspace{.2cm}

\noindent with the right side known to equal to $2g{2g\choose g}$. In fact, the verification of (\ref{eq:double}) would together with the inequalities of Theorem \ref{theorem:main} imply the conjectural recursive equalities.\\

\begin{figure}[t]
  \centering{
  \setlength{\extrarowheight}{-1pt}
  {\small{
  \begin{tabularx}{.42\textwidth}{ccccccc}
  \multicolumn{7}{l}{{\emph{$\Z/2$ Betti numbers of $N_g^\#$}}}\tabularnewline
\multicolumn{7}{l}{}\tabularnewline
$g \; =$   & $\phantom{2}$1$\phantom{2}$ & $\phantom{2}$2$\phantom{2}$ & $\phantom{2}$3$\phantom{2}$  & $\phantom{2}$4$\phantom{2}$  & $\phantom{2}$5$\phantom{2}$   & $\phantom{2}$6$\phantom{2}$   \tabularnewline \midrule
           & 1 & 1 & 1  & 1  & 1   & 1   \tabularnewline
           & \textcolor{purple}{1} & 0 & 0  & 0  & 0   & 0   \tabularnewline 
           &   & 1 & 1  & 1  & 1   & 1   \tabularnewline
           &   & \textcolor{purple}{5} & 6  & 8  & 10  & 12  \tabularnewline 
           &   & \textcolor{purple}{5} & 1  & 1  & 1   & 1   \tabularnewline 
           &   &   & \textcolor{purple}{7}  & 8  & 10  & 12  \tabularnewline
           &   &   & \textcolor{purple}{22} & 29 & 46  & 67  \tabularnewline
           &   &   & \textcolor{purple}{22} & \textcolor{purple}{9}  & 10  & 12  \tabularnewline 
           &   &   &    & \textcolor{purple}{37} & 46  & 67  \tabularnewline
           &   &   &    & \textcolor{purple}{93} & \textcolor{purple}{131} & 232 \tabularnewline 
           &   &   &    & \textcolor{purple}{93} & \textcolor{purple}{56}  & 67  \tabularnewline 
           &   &   &    &    & \textcolor{purple}{176} & \textcolor{purple}{233} \tabularnewline
           &   &   &    &    & \textcolor{purple}{386} & \textcolor{purple}{574} \tabularnewline
           &   &   &    &    & \textcolor{purple}{386} & \textcolor{purple}{299} \tabularnewline
           &   &   &    &    &     & \textcolor{purple}{794} \tabularnewline
           &   &   &    &    &     & \textcolor{purple}{1586}\tabularnewline
           &   &   &    &    &     & \textcolor{purple}{1586} \tabularnewline
\end{tabularx}\qquad\qquad
\begin{tabularx}{.42\textwidth}{ccccccc}
\multicolumn{7}{l}{{\emph{$\Q$ Betti numbers of $N_g^\#$}}}\tabularnewline
\multicolumn{7}{l}{}\tabularnewline
$g \; =$   & $\phantom{2}$1$\phantom{2}$ & $\phantom{2}$2$\phantom{2}$ & $\phantom{2}$3$\phantom{2}$  & $\phantom{2}$4$\phantom{2}$  & $\phantom{2}$5$\phantom{2}$   & $\phantom{2}$6$\phantom{2}$   \tabularnewline \midrule
           & 1 & 1 & 1  & 1  & 1   & 1   \tabularnewline
           & 0 & 0 & 0  & 0  & 0   & 0   \tabularnewline 
           &   & 1 & 1  & 1  & 1   & 1   \tabularnewline
           &   & 4 & 6  & 8  & 10  & 12  \tabularnewline 
           &   & 0 & 1  & 1  & 1   & 1   \tabularnewline 
           &   &   & 6  & 8  & 10  & 12  \tabularnewline
           &   &   & 15 & 29 & 46  & 67  \tabularnewline
           &   &   & 0  & 8  & 10  & 12  \tabularnewline 
           &   &   &    & 28 & 46  & 67  \tabularnewline
           &   &   &    & 56 & 130 & 232 \tabularnewline 
           &   &   &    & 0  & 45  & 67  \tabularnewline 
           &   &   &    &    & 120 & 232 \tabularnewline
           &   &   &    &    & 210 & 561 \tabularnewline
           &   &   &    &    & 0   & 220 \tabularnewline
           &   &   &    &    &     & 495 \tabularnewline
           &   &   &    &    &     & 792 \tabularnewline
           &   &   &    &    &     & 0   \tabularnewline
\end{tabularx}}}
\vspace{.3cm}
}
  \caption{{\small{Comparison of the $\Z/2$ and $\Q$ betti numbers of the framed moduli space. The $\Z/p$ betti numbers for $p$ prime, $p\neq 2$ are the same as the $\Q$ betti numbers. In each column half the betti numbers are listed; the rest are obtained by Poincar\'{e} duality. For example, the $\Z/2$ betti numbers of $N_2^\#$ are 1,0,1,5,5,5,5,1,0,1.}}}\label{fig:introtables}
\end{figure}

The proofs of (i) and (ii) and the inequalities in Theorem \ref{theorem:main} follow an adaptation of Newstead's Mayer-Vietoris argument. We also provide evidence for a stronger statement than the above conjecture, which may be accessible via geometric methods. The framed moduli space embeds into an extended moduli space $N_g^+$ which contains the singular locus $f_g^{-1}(+1)$. If it were the case that the maps on homology induced by inclusion, written in the sequel as

\vspace{.2cm}
\[
    \nu_r^g:H_r(N_g^\#;\Z/2) \longrightarrow H_r(N_g^+;\Z/2),
\]

\vspace{.35cm}

\noindent were always of {\emph{maximal rank}}, then our method would carry through to prove that equality holds in Theorem \ref{theorem:main}. More precisely, we suspect that $\nu_r^g$ is surjective for the first half of the $6g-6$ degrees, and injective for the latter half. We will show that $\nu_r^g$ is of maximal rank for all $r$ when $g\in\{1,2\}$, although we will only sketch our computations in the $g=2$ case. The manifold $N_g^+$ may be viewed as a real algebraic deformation of the singular locus $f_g^{-1}(+1)$ with generic fiber homeomorphic to $N_g^\#$, and understanding $\nu_r^g$ seems an interesting problem in itself.\\

If the conjectural recursive formulae hold, then $H^\ast(N_g^\#;\Z)$ is torsion-free in the first 1/3 and last 1/3 of its degrees, and has nontrivial 2-torsion in-between. We can say a bit more about this. It has been mentioned above that our conjectural formulae have been verified for $g\leqslant 6$ using the Leray-Serre spectral sequence and the computations of \cite{ss-mod2}. In that paper, we study the cohomology ring $H^\ast(N_g;\Z/2)$, and a featured result is that the nilpotency degree of $\alpha\in H^2(N_g;\Z/2)$ is equal to $g$. This latter point is related to the current work as follows. Consider the Bockstein homomorphism associated to the short exact coefficient sequence $\Z\to \Z\to \Z/2$, written

\vspace{.3cm}

\[
    \beta:H^r(N^\#_g;\Z/2) \longrightarrow H^{r+1}(N^\#_g;\Z).
\]

\vspace{.35cm}

\noindent Using a straightforward induction argument, the conjectural formulae imply that the $\Z/2$ betti numbers and $\Q$ betti numbers of the framed moduli space agree up to degree $r=2g-2$. Thus we expect that $\beta=0$ in degrees $r\leqslant 2g-2$. Let $y\in H^1(SO(3);\Z/2)$ be a generator. Then $\alpha^{g-1}\otimes y$ is an element in the $E_2$ page of the Leray-Serre spectral sequence for the fibration $N_g^\#$. By Proposition \ref{prop:ss} below and the nilpotency $\alpha^g=0$ from \cite[Thm. 1]{ss-mod2}, it survives to the $E_\infty$ page to define a non-zero element $[\alpha^{g-1}\otimes y]\in H^{2g-1}(N_g^\#;\Z/2)$. This element has no integral lift since $y$ has no integral lift, and thus we obtain the following.

\vspace{.4cm}

\begin{corollary}\label{cor:bockstein}
$\beta\left([\alpha^{g-1}\otimes y]\right)\neq 0$.
\end{corollary}

\vspace{.4cm}

\noindent From the discussion above, we expect this to account for the first difference between the $\Z/2$ betti numbers and $\Q$ betti numbers, which occurs at $r=2g-1$. In fact, the conjectural formulae imply that the $\Z/2$ betti number at $r=2g-1$ is always exactly one more than the $\Q$ betti number, and thus we expect that the element $[\alpha^{g-1}\otimes y]$ entirely accounts for this difference.\\

Finally, we make a few remarks on other approaches to proving equality in \ref{eq:rec1}$-$\ref{eq:rec3}. One might try to apply Thaddeus's Morse-theoretic argument of \cite{thaddeus-morse} to the framed moduli space. Indeed, a priori, the function $(A_i,B_i)\longmapsto \text{tr}(A_g)$ defined on $N_g^\#$, the pullback of Thaddeus's function, may be perfect Morse-Bott over $\Z/2$. This is not the case, however: for genus 2, the betti numbers for the starting page of the Bott-Morse spectral sequence with $\Z/2$ coefficients are $1,1,2,6,6,6,6,2,1,1$, while the $\Z/2$ betti numbers of $N_2^\#$, which constitute the $E_\infty$ page, are $1,0,1,5,5,5,5,1,0,1$. The gaps between these pages increases as the genus grows. On a related note, it would be interesting to see if the $\infty$-dimensional method of Atiyah-Bott \cite{ab} has anything to say here.\\

\vspace{5cm}

\noindent \textbf{Outline.} In Section \ref{sec:preliminaries} we fix our notation and record some useful results from \cite{newstead-mv}. In Section \ref{sec:genus1} we compute some data in the genus 1 case in order to apply Newstead's Mayer-Vietoris argument in Section \ref{sec:mainarg} to prove Theorem \ref{theorem:main}. Finally, in Section \ref{sec:genus2} we sketch the arguments that show $\nu_r^g$ is of maximal rank for genus $2$.\\

\vspace{.7cm}

\noindent \textbf{Acknowledgments.} The first author thanks Simon Donaldson and Ali Daemi for encouraging conversations. The first author was supported by NSF grant DMS-1503100.

\newpage

\section{Preliminaries}\label{sec:preliminaries}

In this section we list some facts from Newstead's paper \cite{newstead-mv} and fix notation and conventions. All homology groups will be with $\F=\Z/2$ coefficients unless otherwise indicated, and we write $|V|$ for the dimension of a vector space $V$. Although we henceforth fix our coefficient field $\F$, it is worth remarking that the results of this section hold for any coefficient field.\\

Write $SU(2) = D_+ \cup D_-$ as a union of two 3-balls, each with boundary the 2-sphere of trace-free elements, and with $\pm 1\in D_\pm$. Then define the $6g$-dimensional manifolds with boundary

\vspace{.2cm}
\[
	N_g^{\pm} \; = \; f_g^{-1}(D_\pm).
\]
\vspace{.2cm}

\noindent Newstead explains that $N_g^-$ is homeomorphic to $D_-\times N_g^\#$. In particular, the boundaries of both $N_g^+$ and $N_g^-$ may be identified with $S^2\times N_g^\#$. Define the betti numbers

\vspace{.3cm}
\[
	\check{n}^g_r \; = \; |H_r(N_g^+)|, \qquad\qquad  \hat{n}^g_r \; = \; |H_r(N_g^+,\partial N_g^+)|.
\]
\vspace{.3cm}

\noindent Note that $\check{n}_{6g-r}^g = \hat{n}^g_r$ by Lefschetz duality. Let $\mu_r^g:H_r(\partial N_g^+) \longrightarrow H_r(N_g^+)$ be the map on homology induced by inclusion. Using the K\"{u}nneth decomposition for the homology of the boundary of $N_g^+$, we can write $\mu_r^g$ as the sum of two maps, $\nu_r^g$ and $\rho_r^g$:

\vspace{.25cm}
\begin{center}
\begin{tikzcd}[column sep=large] 
    H_0(S^2)\otimes H_r(N_g^\#)  \arrow[to=1-2, "\nu_r^g"] & H_r(N_g^+)\\
    H_2(S^2)\otimes H_{r-2}(N_g^\#) \arrow[to=1-2, "\rho_r^g"', start anchor=east] & 
\end{tikzcd}
\end{center}
\vspace{.25cm}

\noindent Note that the domains of $\nu_r^g$ and $\rho_r^g$ are naturally isomorphic to $H_r(N_g^\#)$ and $H_{r-2}(N_g^\#)$, respectively. These two maps play a central role in the sequel.
Write $m_r^g$ for the betti numbers of $SU(2)^{2g}$. These were given in the introduction as the coefficients of $(1+t^3)^{2g}$. They are explicitly given by:

\vspace{.2cm}
\[
    m_r^g \; = \; \dim H_r(SU(2)^{2g}) \; =\; \begin{cases} \displaystyle{{2g\choose r/3}}, & r\equiv 0\, (\text{mod } 3)\\ & \\ \quad 0\;, & \text{otherwise} \end{cases}
\]
\vspace{.2cm}

\noindent We now list some elementary relations between the quantities thus far introduced. To start, the following says that the betti numbers of $N_g^+$ determine those of $N_g^\#$ and conversely:

\vspace{.56cm}
\begin{lemma}\label{lemma:n}
\[
	\check{n}_r^g \; = \; \begin{cases} h_{r-2}^g + m_r^g &  (r\leqslant 3g+1)\\  h_{r-2}^g - m_{r+1}^g  &  (r \geqslant 3g+1)\end{cases} \qquad \qquad \hat{n}_r^g \; = \; \begin{cases} h_{r-1}^g - m_{r-1}^g &  (r\leqslant 3g-1)\\  h_{r-1}^g + m_{r}^g &  (r\geqslant 3g-1)\end{cases}
\]
\end{lemma}

\vspace{.6cm}

\noindent This lemma follows from Lemmas 2 and 3 in Section 7 of \cite{newstead-mv}. There, Newstead shows that the two maps $H_r(N_g^+)\longrightarrow H_r(SU(2)^{2g})$ and $H_r(N_g^\#)\longrightarrow H_r(SU(2)^{2g})$ induced by inclusion are surjective for $r\leqslant 3g+2$ and $r\leqslant 3g-1$, respectively. His arguments for surjectivity are elementary and easily seen to hold for any coefficient ring. The formula for $\check{n}_r^g$ with $r\leqslant 3g+1$ then follows by looking at the long exact sequence associated to the pair $(SU(2)^{2g},N_g^+)$ and observing that excision identifies the group $H_r(SU(2)^{2g},N_g^+)$ with $H_{r-2}(N_g^\#)$. The formula for $\hat{n}_r^g$ with $r\leqslant 3g-1$ follows in a similar way, and the rest of the formulae follow by Lefschetz duality.\\

Next, we mention that the kernels and cokernels of the maps $\rho_r^g$ and $\mu_r^g$ are also determined by the betti numbers of $N_g^\#$. From the long exact sequence of the pair $(N_g^+,\partial N_g^+)$ we have

\vspace{.2cm}
\[
	|\text{coker}(\mu_r^g)| + |\text{ker}(\mu_{r-1}^g)| \; = \; |H_r(N_g^+,\partial N_g^+)| \; = \; \hat{n}_r^g,
\]
\vspace{.2cm}

\noindent from which the following is easily computed, with help of the above lemma:

\vspace{.2cm}
\[
	|\text{ker}(\mu_r^g)| \; = \; 
	\begin{cases} h_r^g - m_r^g &  (r < 3g)\\  
	           h_{r}^g + m_{r+1}^g  &  (r \geqslant 3g)
	\end{cases}
\]
\vspace{.2cm}

\noindent And for the map $\rho_r^g$ we may consider the Mayer-Vietoris sequence associated to the decomposition of $SU(2)^{2g}$ into the union of $N_g^+$ and $N_g^-$ along their boundaries:

\vspace{.2cm}
\begin{center}
\begin{tikzcd}[column sep=large] 
    \cdots \;\;\; H_r(S^2\times N_g^\#)  \arrow[to=1-2, "\chi^g_r"] & H_r(N_g^+)\oplus H_r(N_g^-) \arrow[to=1-3] & H_r(SU(2)^{2g}) \;\;\;\cdots
\end{tikzcd}
\end{center}
\vspace{.2cm}

\noindent The fact that $N_g^-$ is homeomorphic to $D^3\times N_g^\#$ implies that the kernel and cokernel of $\chi^g_r$ are isomorphic to the kernel and cokernel of $\rho^g_r$, respectively. From this we have

\vspace{.2cm}
\[
	|\text{coker}(\rho_r^g)| + |\text{ker}(\rho_{r-1}^g)| \; = \; |H_r(SU(2)^{2g})| \; = \; m_r^g.
\]
\vspace{.2cm}

\noindent Solving for the kernel and cokernel of $\rho_r^g$ amounts to the following very useful observation:

\vspace{.5cm}
\begin{lemma}\label{lemma:rho}
{$\phantom{3}$}\\
\vspace{-.3cm}
\begin{enumerate}
    \item If $r\leqslant 3g + 1$, then $\rho_r^g$ is injective, and its cokernel has dimension $m_r^g$. In particular, if also $r \equiv 1,2\, (\text{{\emph{mod}} }3)$, then $\rho_r^g$ is an isomorphism.\\
    \item If $r\geqslant 3g + 1$, then $\rho_r^g$ is surjective, and its kernel has dimension $m_{r+1}^g$. In particular, if also $r \equiv 0,1\, (\text{{\emph{mod}} }3)$, then $\rho_r^g$ is an isomorphism.
\end{enumerate}
\end{lemma}
\vspace{.5cm}

\noindent We do not have as easy a way to compute the kernels and cokernels of the maps $\nu_r^g$ in general. We will determine these quantities for low genus examples.

\vspace{.6cm}

\newpage

\section{Getting started with the genus 1 decomposition}\label{sec:genus1}

Now we begin the adaptation of Newstead's Mayer-Vietoris argument with coefficients in $\F$. It is from this point onwards that the situation differs from the case of a field that has characteristic not equal to 2. We begin by decomposing, as does Newstead, the genus $g+1$ framed moduli space into two parts that are built from genus $1$ and genus $g$ data:

\vspace{.2cm}
\begin{equation}\label{eq:mv1}
	 N_{g+1}^\# \; = \;  N_1^+ \times N_g^\# \bigcup_{S^2\times N_1^\# \times N_g^\#} N_1^\# \times N_g^+
\end{equation}
\vspace{.2cm}

\noindent We refer to \cite[\S 4]{newstead-mv} for details. Recall here that $N_1^\#$ may be identified with $SO(3)$, with betti numbers $1,1,1,1$, and from Lemma \ref{lemma:n}, those of $N_1^+$ are $1,0,1,3,1$. We can then fill in most of the data for the maps we considered in the previous section with $g=1$ in the following table: 

\vspace{.3cm}

\begin{figure}[h]
  \centering{
  \setlength{\extrarowheight}{7pt}
\begin{tabularx}{0.65\textwidth}{c |*{6}{Y}}
$r$        & $h_r^1$ & $\check{n}_r^1$ & $\mu_{r}^1$ & $ \rho_{r}^1$ & $\nu_{r}^1$ \tabularnewline \midrule
0           & 1 & 1 & $1_1^1$ & $0_0^1$ &  $1_1^1$\tabularnewline
1           & 1 & 0 & $0_1^0$ & $0_0^0$ &  $0_1^0$ \tabularnewline 
2           & 1 & 1 & $1_2^1$ & $1_1^1$ &   \fbox{$1_1^1$}\tabularnewline
3           & 1 & 3 & $1_2^3$ & $1_1^3$ &   \fbox{$1_1^3$}\tabularnewline 
4           & 0 & 1 & $1_1^1$ & $1_1^1$ & $0_0^1$\tabularnewline 
5           & 0 & 1 & $0_1^0$ & $0_1^0$ & $0_0^0$\tabularnewline 
\end{tabularx}
\vspace{.3cm}
}
  \caption{{\small{Genus 1 data. The notation $a_b^c$ stands for a linear map $\F^b\longrightarrow \F^c$ of rank $a$.  All entries are computed from the first column from relations in Section \ref{sec:preliminaries}, except for $\nu_2^1$ and $\nu_3^1$ (boxed) -- see Lemma \ref{lemma:nu1}.}}}
\end{figure}

\vspace{.3cm}

\noindent In fact, all of this data (not including $\nu_2^1$ and $\nu_3^1)$ can be deduced from Newstead's table \cite[\S 5]{newstead-mv} via universal coefficients. Now consider the Mayer-Vietoris sequence corresponding to (\ref{eq:mv1}):

\vspace{.2cm}
\begin{center}
\begin{tikzcd}[column sep=large] 
    \cdots \;\;\; H_r(S^2\times N_1^\#\times N_g^\#)  \arrow[to=1-2, "\lambda^{1,g}_r"] & H_r(N_1^+\times N_g^\#)\oplus H_r(N_1^\#\times N_g^+) \arrow[to=1-3] & H_r(N_{g+1}^\#) \;\;\;\cdots
\end{tikzcd}
\end{center}
\vspace{.2cm}

\noindent Then the exactness of the Mayer-Vietoris sequence yields the following:

\vspace{.2cm}
\begin{equation}
	h_r^{g+1} \; = \; |\text{coker}(\lambda_r^{1,g})| + |\text{ker}(\lambda_{r-1}^{1,g})|\label{eq:lambda}
\end{equation}
\vspace{.2cm}

\noindent To understand $\lambda_r^{1,g}$ we decompose all of the homology groups using the K\"{u}nneth Theorem. Before doing this, let us write the two components of $\lambda_r^{1,g}$ as maps in two different directions: \\

\begin{center}
\begin{tikzcd}[column sep=large] 
H_r(N_1^+\times N_g^\#) &  H_r(S^2\times N_1^\#\times N_g^\#) \arrow[to=1-1, red]  \arrow[to=1-3, blue] & H_r(N_1^\#\times N_g^+) \\
\end{tikzcd}
\end{center}

\noindent Write $\iota^g_r$ for the identity map on $H_r(N_g^\#)$. From here we expand the map $\lambda_r^{1,g}$ using the K\"{u}nneth decompositions of the three homology groups:\\

\vspace{.15cm}

\begin{center}
\begin{tikzcd}[column sep=large] 
H_0(N_1^+)\otimes H_r(N_g^\#) & H_0(S^2)\otimes H_0(N_1^\#)\otimes H_r(N_g^\#) \arrow[to=1-1, red, "\nu^1_0\,\otimes\, \iota^g_r"' black]  \arrow[to=1-3, blue, "\iota_0^1\,\otimes\,\nu^g_r" black] & H_0(N_1^\#)\otimes H_r(N_g^+) \\
H_2(N_1^+)\otimes H_{r-2}(N_g^\#) & H_2(S^2)\otimes H_0(N_1^\#)\otimes H_{r-2}(N_g^\#) \arrow[to=2-1, red, "\rho^1_2\,\otimes\, \iota^g_{r-2}"' black]  \arrow[to=1-3, blue, "\iota_0^1\,\otimes\,\rho^g_r"' black,start anchor=north east, end anchor=south west] & \\

	& H_0(S^2)\otimes H_1(N_1^\#)\otimes H_{r-1}(N_g^\#)  \arrow[to=3-3, blue, "\iota_1^1\,\otimes\,\nu^g_{r-1}" black] & H_1(N_1^\#)\otimes H_{r-1}(N_g^+) \\
H_3(N_1^+)\otimes H_{r-3}(N_g^\#) & H_2(S^2)\otimes H_1(N_1^\#)\otimes H_{r-3}(N_g^\#) \arrow[to=4-1, red, "\rho^1_3\,\otimes\, \iota^g_{r-3}"' black]  \arrow[to=3-3, blue, "\iota_1^1\,\otimes\,\rho^g_{r-1}"' black,start anchor=north east, end anchor=south west] & \\

	& H_0(S^2)\otimes H_2(N_1^\#)\otimes H_{r-2}(N_g^\#)  \arrow[to=2-1, red, "\nu^1_2\,\otimes\, \iota^g_{r-2}" black, near end, start anchor=north west, end anchor=south east] \arrow[to=5-3, blue, "\iota_2^1\,\otimes\,\nu^g_{r-2}" black] & H_2(N_1^\#)\otimes H_{r-2}(N_g^+) \\
H_4(N_1^+)\otimes H_{r-4}(N_g^\#) & H_2(S^2)\otimes H_2(N_1^\#)\otimes H_{r-4}(N_g^\#) \arrow[to=6-1, red, "\rho^1_4\,\otimes\, \iota^g_{r-4}"' black]  \arrow[to=5-3, blue, "\iota_2^1\,\otimes\,\rho^g_{r-2}"' black,start anchor=north east, end anchor=south west] &\\

	& H_0(S^2)\otimes H_3(N_1^\#)\otimes H_{r-3}(N_g^\#)  \arrow[to=4-1, red, "\nu^1_3\,\otimes\, \iota^g_{r-3}" black, near end, start anchor=north west, end anchor=south east] \arrow[to=7-3, blue, "\iota_3^1\,\otimes\,\nu^g_{r-3}" black] & H_3(N_1^\#)\otimes H_{r-3}(N_g^+) \\
	& H_2(S^2)\otimes H_3(N_1^\#)\otimes H_{r-5}(N_g^\#) \arrow[to=7-3, blue, "\iota_3^1\,\otimes\,\rho^g_{r-3}"' black,start anchor=north east, end anchor=south west] &\\
\end{tikzcd}
\end{center}

\vspace{.15cm}

\noindent Note that each homology group of $S^2$, $N_1^\#$ and $N_1^+$ that appears here is isomorphic to $\F$, with the exception of $H_3(N_1^+)$, which is rank 3. In the sequel, it will be convenient to replace each vector space that appears in such a diagram by a dot $\bullet$ as in Figure \ref{fig:dots2}.

Now, if we plug $r=0,1,2$ into this diagram, the kernels and cokernels are easy to compute with what we know thus far; for example, see Figure \ref{fig:dots1}. We obtain the following:

\vspace{.2cm}
\[
    |\text{coker}(\lambda_0^{1,g})|  \; = \; |\text{coker}(\lambda_2^{1,g})| \; = \; |\text{ker}(\lambda_2^{1,g})| \; = \;  1,
\]
\vspace{.1cm}
\[
    |\text{ker}(\lambda_0^{1,g})| \; = \; |\text{ker}(\lambda_1^{1,g})| \; = \; |\text{coker}(\lambda_1^{1,g})| \; = \; 0.
\]
\vspace{.2cm}

\noindent Using equation (\ref{eq:lambda}) we then deduce, for all $g\geqslant 2$, that $h_0^g=1$, $h_1^g = 0$ and $h_2^g=1$. The first two of these equalities alternatively follow from Newstead's Theorem 1 \cite{newstead-mv}, which says that the framed moduli space is simply connected for $g\geqslant 2$. 

In trying to compute the kernel of the next map $\lambda_3^{1,g}$ to determine $h_3^g$, we find that the answer depends on $\nu_2^1$, which we have not yet determined. To help solve for the map $\nu_2^1$ we will look at the genus 2 moduli space. Before proceeding with this, we make a short digression regarding the Leray-Serre spectral sequence for the framed moduli space.

The cohomological Leray-Serre spectral sequence for the $SO(3)$-fibration $N_g^\#$ with base space $N_g$ is depicted in Figure \ref{fig:ss2} for $g=2$, the details of which will be explained shortly. Write $y$ for the degree 1 generator of $H^\ast(SO(3))$. Now recall $H_1(N_2)=0$; in fact, $N_g$ is simply connected \cite[Cor. 1]{newstead-mv}. Since also $h_1^g=0$ from above, the $d_2$ differential on the $E_2$-page of the spectral sequence must be non-zero on the element $1\otimes y$. Thus $d_2(1\otimes y)=\alpha\otimes 1$, and using the Leibniz rule, we obtain that for any $x\in H^\ast(N_2)$ we have $d_2(x\otimes y^i)=\alpha x\otimes y^{i-1}$ for $i\in\{1,3\}$, and $d_2$ is otherwise $0$.

From here, the only possible element in $E_2$ to survive to $H^2(N_2^\#)$ is represented by $1\otimes y^2$. However, we already computed above that $h_2^2=1$, necessitating its survival. Thus the $E_i$-page differential $d_i$ for $i\geqslant 3$ is zero on the class of $1\otimes y^2$. Since $d_i$ for $i\geqslant 3$ vanishes on the bottom two rows of the $E_i$-page for degree reasons, and every element in the top two rows is a multiple of the class of $1\otimes y^2$, the Leibniz rule implies that $d_i$ vanishes everywhere. Thus we have:

\begin{figure}[t]
\centering{
\begin{tikzpicture}
  \matrix (m) [matrix of math nodes,
    nodes in empty cells,nodes={minimum width=10ex,
    minimum height=5ex,outer sep=-5pt},
    column sep=1ex,row sep=1ex]{
	&&&&&\\
	3 & 1\otimes y^3 & 0  &  \alpha\otimes y^{3} &  \langle\psi_i\rangle\otimes y^3 &  \delta_2\otimes y^3 &  0 &  \alpha\delta_2\otimes y^3 & \\
          2    &  1\otimes y^2    &   0  &  \alpha\otimes y^{2}   &  \langle\psi_i\rangle\otimes y^2 & \delta_2\otimes y^2 & 0 & \alpha\delta_2\otimes y^2 & \\
          1     &  1\otimes y^{\phantom{2}} &  0  & \alpha\otimes y^{\phantom{2}} &  \langle\psi_i\rangle\otimes y^{\phantom{2}} & \delta_2 \otimes y^{\phantom{2}}& 0 &  \alpha\delta_2\otimes y^{\phantom{2}} &\\
          0     &  1\otimes 1^{\phantom{2}}  & 0 &  \alpha\otimes 1^{\phantom{2}}  & \langle\psi_i\rangle\otimes 1^{\phantom{2}} &  \delta_2\otimes 1^{\phantom{2}} & 0 & \alpha\delta_2\otimes 1^{\phantom{2}} &\\
    \quad\strut &   0  &  1  &  2  & 3 & 4 & 5 & 6 & \strut \\};
 \draw[-stealth] (m-4-2.south east) -- (m-5-4.north west);
 \draw[-stealth] (m-2-2.south east) -- (m-3-4.north west);
 \draw[-stealth] (m-4-6.south east) -- (m-5-8.north west);
 \draw[-stealth] (m-2-6.south east) -- (m-3-8.north west);
\draw[thick] (m-1-1.east) -- (m-6-1.east) ;
\draw[thick] (m-6-1.north) -- (m-6-9.north) ;
\end{tikzpicture}}
\caption{The $E_2$-page in the Leray-Serre spectral sequence for $N_2^\#$}\label{fig:ss2}
\end{figure}
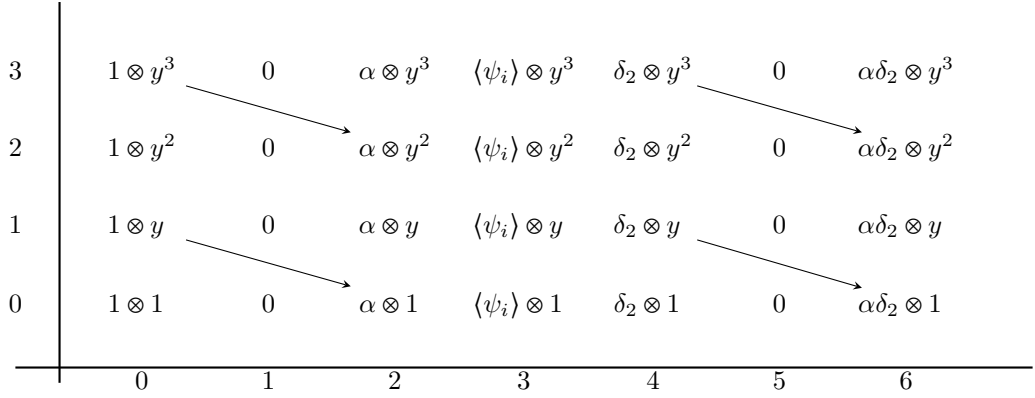

\vspace{.5cm}
\begin{prop}\label{prop:ss}
For $g\geqslant 2$, the $E_2$-page differential in the cohomological Leray-Serre spectral sequence for $N_g^\#$ sends $x\otimes y^i$ to the element $\alpha x \otimes y^{i-1}$ for $i\in \{1,3\}$ and any $x\in H^\ast(N_g)$, and is otherwise zero. The spectral sequence collapses at the $E_3$-page. Consequently, we have the formula

\vspace{.25cm}
\[
    h_r^g \; = \; |\text{{\emph{coker}}}(\alpha_{r-2}^g)| + |\text{{\emph{ker}}}(\alpha_{r-1}^g)|  + |\text{{\emph{coker}}}(\alpha_{r-4}^g)| + |\text{{\emph{ker}}}(\alpha_{r-3}^g)|
\]
\vspace{.25cm}

\noindent where $\alpha^g_r:H_r(N_g)\longrightarrow H_{r+2}(N_g)$ is the map defined by cup product with $\alpha$.
\end{prop}
\vspace{.5cm}

Now we explain the genus 2 case more fully. The moduli space $N_2$ is 6-dimensional, and its cohomology ring over $\F$ is generated be a degree 2 element $\alpha$, degree 3 elements $\psi_1,\psi_2,\psi_3,\psi_4$, and a degree 4 element $\delta_2$. The ring structure is determined by the following: the only top degree monomials that pair nontrivially with the fundamental class $[N_2]$ are the following:

\vspace{.2cm}
\[
    \alpha\delta_2, \quad \psi_1\psi_3, \quad \psi_2\psi_4.
\]
\vspace{.2cm}

\noindent In particular, $\alpha^2=0$. This ring, and in fact the corresponding ring with integer coefficients, is described in Remark 2 of Section 10 in \cite{newstead-mv}. 

Now Figure \ref{fig:ss2} is obtained from this description of the ring and Proposition \ref{prop:ss}. The arrows drawn represent the non-trivial $E_2$ differentials. Note that we have written $\langle \psi_i \rangle$ for the 4-dimensional vector space with basis the $\psi_i$ classes. The numbers $h_r^2$ are then computed from Figure \ref{fig:ss2} to be

\vspace{.2cm}
\begin{equation}
    1, \, 0,\, 1,\, 5,\,5,\,5,\,5,\, 1,\,0,\,1.\label{eq:g2betti}
\end{equation}
\vspace{.2cm}

\noindent We are now in a position to compute $\nu_2^1$. Consider the map $\lambda_3^{1,1}$. Referring to Figure \ref{fig:dots1}, we find that the cokernel of this map is $4$ or $6$, depending on whether $\nu_2^1$ is an isomorphism or not, respectively. Since we now know that $h_3^2=5$, and from above $|\text{ker}(\lambda_2^{1,1})|=1$, equation (\ref{eq:lambda}) implies that $\nu_2^1$ must in fact be an isomorphism. We now have the first part of:

\begin{figure}[t]
  \centering{
\includegraphics[scale=1]{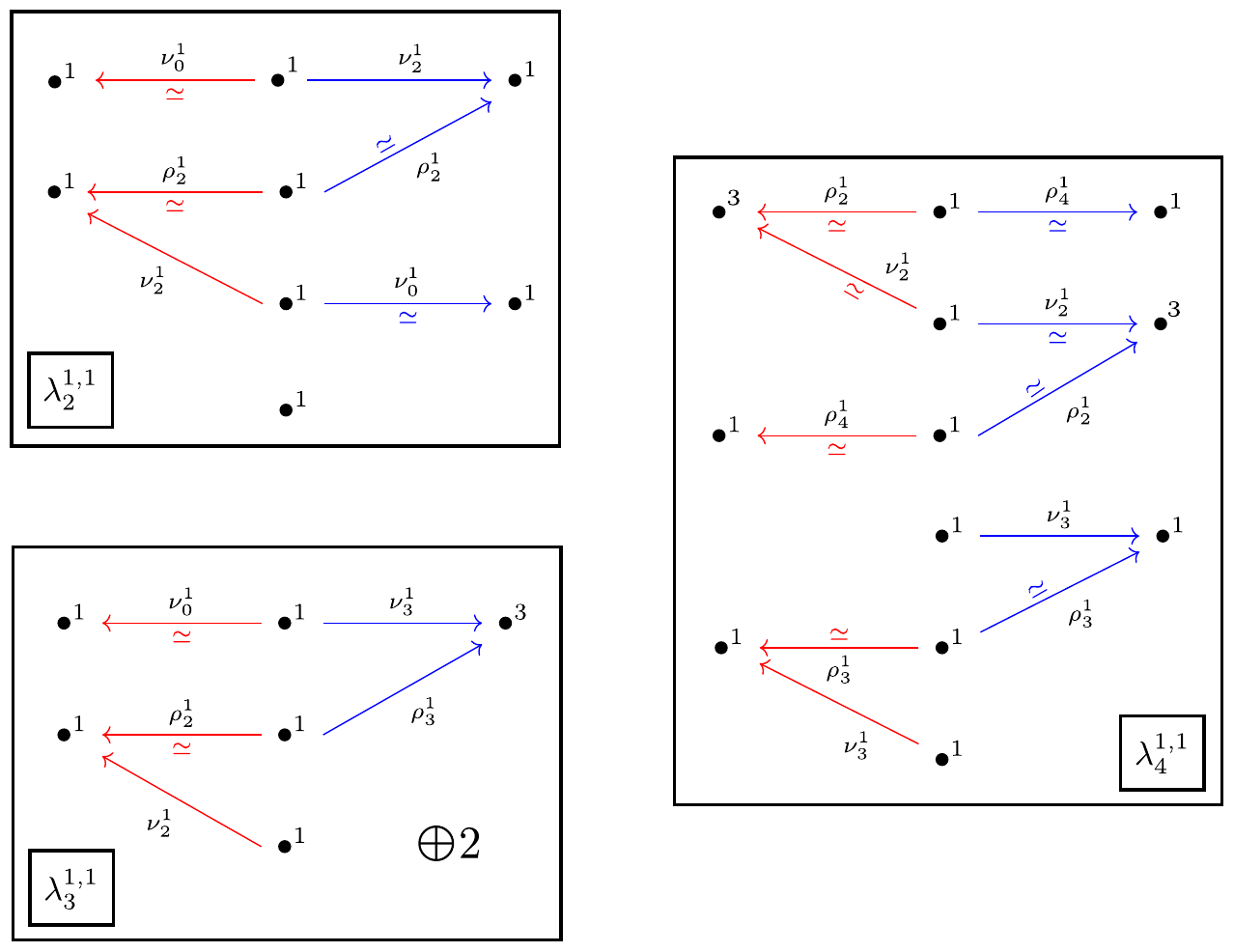}
\vspace{.2cm}
}
  \caption{{\small{The maps $\lambda_2^{1,1}$, $\lambda_3^{1,1}$ and $\lambda_4^{1,1}$. Each vector space has been replaced by a dot $\bullet$. The dimension of each vector space is written as a superscript of each $\bullet$. The notation $\oplus 2$ in the lower left pane indicates that the map $\lambda_3^{1,1}$ consists of two copies of the depicted map. The lone lower dot in the upper left pane of $\lambda_2^{1,1}$ comes from the domain of $\iota_1^1\otimes\nu_1^1$.}}}\label{fig:dots1}
\end{figure}

\vspace{.6cm}
\begin{lemma}\label{lemma:nu1}
The maps $\nu_2^1$ and $\nu_3^1$ are injective.
\end{lemma}
\vspace{.6cm}

\noindent To compute $\nu_3^1$ we next consider $\lambda_4^{1,1}$. Referring again to Figure \ref{fig:dots1}, we find that the cokernel of $\lambda_4^{1,1}$ has dimension equal to 5 or 6 depending on whether $\nu_3^1$ is injective or not, respectively. We are using our knowledge that the image of $\nu_3^1$ is contained in that of $\rho_3^1$, as follows from $\mu_3^1$ having rank 1. From above, the kernel of $\lambda_3^{1,1}$ has dimension 0. Finally, from (\ref{eq:g2betti}) we have $h_4^2=5$, and this forces via (\ref{eq:lambda}) the dimension of the cokernel of $\lambda_4^{1,1}$ to be 5, implying that $\nu_3^1$ is injective. This completes the proof of the lemma.

\vspace{.7cm}

\section{Applying the Mayer-Vietoris argument}\label{sec:mainarg}

With all of the genus 1 data computed, we are now in a position to prove Theorem \ref{theorem:main}. Referring to Figure \ref{fig:dots2}, we first replace $\lambda_r^{1,g}$ with a map $\psi_r^{1,g}$ that has the same kernel and cokernel. We will shortly focus on this latter map.

Going from $\lambda_r^{1,g}$ to its simplification $\psi_r^{1,g}$ is only a matter of linear algebra over $\F$. In fact, from the diagrammatic perspective, it is a standard manipulation in the context of computing homology groups over $\F$, usually referred to there as Gaussian elimination. For example, when an arrow is an isomorphism and no other arrow touches its codomain, then we can eliminate the arrow, along with its domain and codomain. This rule allows us to erase from $\lambda_r^{1,g}$ the top left arrow $\nu_0^1\otimes \iota_r^g$ as well as the arrow corresponding to $\rho_4^1\otimes\iota_{r-4}^g$. Next, the fact that $\nu_2^1$ and $\rho_2^1$ are isomorphisms allows us to join the domains of $\iota_0^1\otimes\rho_r^g$ and $\iota_2^1\otimes\nu_{r-2}^g$. We can do the same for $\nu_3^1$ and $\rho_3^1$ to join the domains of $\iota_1^1\otimes\rho_{r-1}^g$ and $\iota_3^1\otimes\nu_{r-3}^g$, except that $\nu_3^1$ and $\rho_3^1$ are only isomorphisms onto their common images: we must also save a complement of this image in their codomain, which will be of dimension $2h_{r-1}^g$. The result after doing these manipulations is the diagram defining $\psi_r^{1,g}$.

\begin{figure}[t]
  \centering{
\includegraphics[scale=1.05]{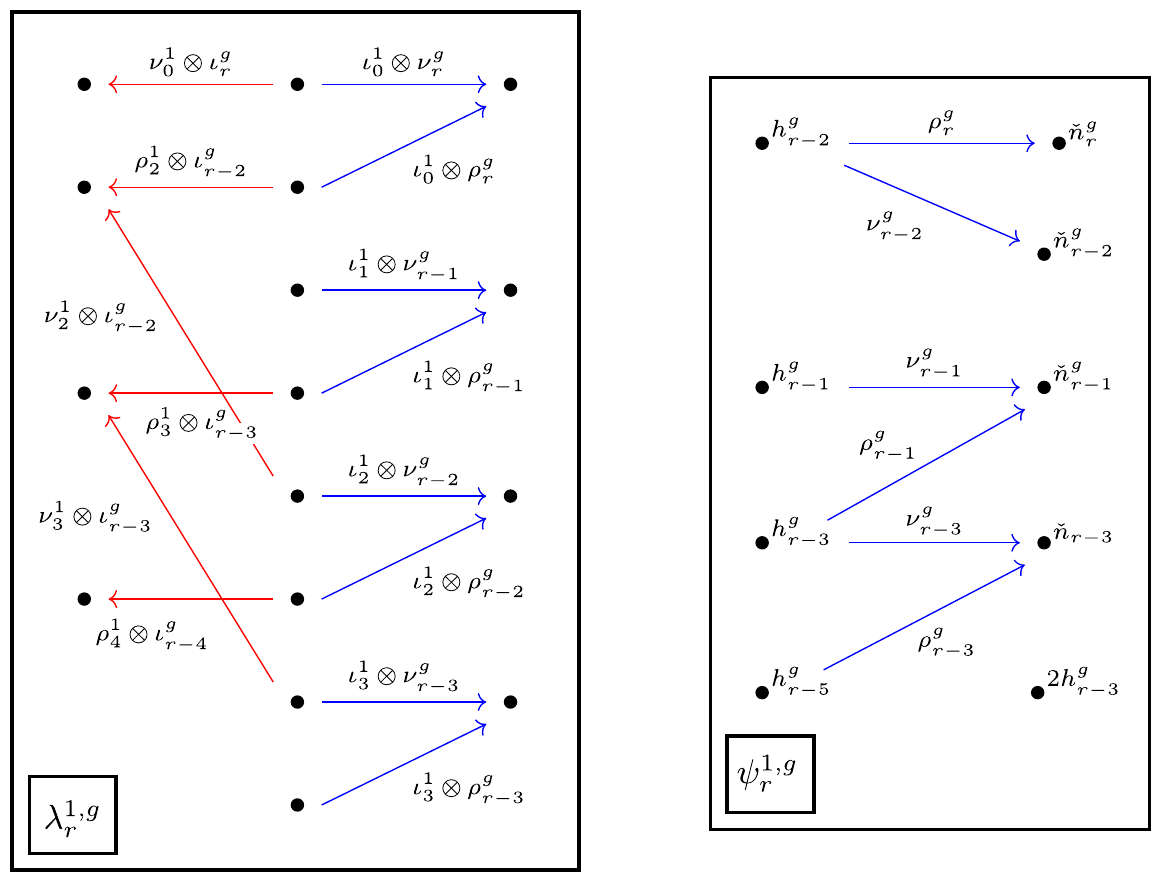}
}
\vspace{.2cm}
\caption{{\small{In the left hand pane, we have simply redrawn the above expansion of $\lambda_r^{1,g}$ with a dot $\bullet$ replacing the name of each vector space. The computation of all the left hand (red) maps in this pane allows us to replace $\lambda_r^{1,g}$ with the map $\psi_r^{1,g}$ defined in the right hand pane.}}}\label{fig:dots2}
\end{figure}

Now Lemma \ref{lemma:rho} allows us to compute the kernel and cokernel of $\psi_r^{1,g}$ in many cases. For example, suppose that $r\leqslant 3g+1$ and $r\equiv 1,2 $ (mod 3). Then the part of the map consisting of $\rho_r^g$ and $\nu_{r-2}^g$ in the diagram for $\psi_r^{1,g}$ does not contribute to the kernel, since $\rho_r^g$ is injective. Also, $\rho_r^g$ is an isomorphism, so it along with its domain and codomain can be eliminated from consideration. After this, as far as the kernel goes, we are left only with the part of the map consisting of $\nu_{r-1}^g$ and $\rho_{r-1}^g$, which is exactly $\mu_{r-1}^g$. For $r\leqslant 3g$ we have $|\text{ker}(\mu_{r-1}^g)|=h_{r-1}^g-m_{r-1}^g$, while for $r=3g+1$ we have instead  $|\text{ker}(\mu_{3g}^{g})| = h_{3g}^g$. We have deduced the first two parts of:

\vspace{.6cm}
\begin{lemma}
${}$\label{lemma:lambdaker}
\begin{enumerate}

\item If $r < 3g+1$ and $r\equiv 1,2$ {\emph{(mod 3)}}, then $|\text{{\emph{ker}}}(\lambda_r^{1,g})| =  h_{r-1}^g-m_{r-1}^g$.

\item If $r=3g+1$ then $|\text{{\emph{ker}}}(\lambda_{3g+1}^{1,g})| = h_{3g}^g$.

\item If $r\geqslant 3g+4$ and $r\equiv 0,1$ {\emph{(mod 3)}}, then $|\text{{\emph{ker}}}(\lambda_r^{1,g})| = h_{r-1}^g + m_r^g$.

\item If $r < 3g+1$ and $r\equiv 1,2$ {\emph{(mod 3)}}, then $|\text{{\emph{coker}}}(\lambda_r^{1,g})| = 2h_{r-3}^g +h^g_{r-4}  +m^g_{r-2}$.

\item If $r=3g+1$ then $|\text{{\emph{coker}}}(\lambda_{3g+1}^{1,g})| = 2h^g_{3g-2} + h_{3g-3}^g +m_{3g}^g$.

\item If $r\geqslant 3g+4$ and $r\equiv 0,1$ {\emph{(mod 3)}}, then $|\text{{\emph{coker}}}(\lambda_r^{1,g})| = 2h_{r-3}^g +h_{r-4}^g-m_{r-1}^g$.

\end{enumerate}
\end{lemma}
\vspace{.6cm}

\noindent The third item in the lemma is proven similarly: in this range, $\rho_{r}^g$ is an isomorphism, so again the top part of the diagram for $\psi_r^{1,g}$ contributes no kernel. The map $\rho^g_{r-3}$ is also an isomorphism, and in the same way as before we identify the kernel of $\psi_r^{1,g}$ with that of $\mu_{r-1}^g$. The only difference is that in this range we have $|\text{ker}(\mu_{r-1}^g)|=h_{r-1}^g+m_{r}^g$. The latter three items of the lemma follow from the first three by simply inspecting the dimensions of the domain and codomain of $\lambda_r^{1,g}$.\\

\begin{proof}[Proof of (i)-(ii) in Thm. \ref{theorem:main}]
Substitute the items of Lemma \ref{lemma:lambdaker} into (\ref{eq:lambda}).
\end{proof}

\vspace{.5cm}

\begin{lemma}\label{lemma:kercap}
For all $r$, the inequalities {\emph{\ref{eq:rec1}$-$\ref{eq:rec3}}} are valid. Equality holds if and only if for all $r$,
\vspace{.2cm}
\begin{equation}
    |\text{{\emph{ker}}}(\rho_r^g)\cap \text{{\emph{ker}}}(\nu_{r-2}^g)| \; = \; 0. \qquad \label{eq:kercap}
\end{equation}
\end{lemma}
\vspace{.3cm}

\begin{proof}
We first note that (\ref{eq:kercap}) holds whenever $\rho_r^g$ is injective. Thus Lemma \ref{lemma:rho} implies (\ref{eq:kercap}) for the ranges $r\geqslant 3g+1$ with $r\equiv 0,1$ (mod 3), and $r\leqslant 3g+1$. We now focus on the cases in which $r\geqslant 3g+2$ and $r\equiv 2$ (mod 3). First suppose $r\geqslant 3g+4$ and $r\equiv 2$ (mod 3). Referring to the diagram for $\psi_r^{1,g}$ in Figure \ref{fig:dots2}, and using the fact that $\rho_{r-3}^g$ is surjective, the kernel is seen to have dimension

\vspace{.2cm}
\begin{equation}\label{eq:kerpsi}
    |\text{ker}(\lambda_r^{1,g})| \; = \; |\text{ker}(\rho_{r-3}^g)| + |\text{ker}(\mu_{r-1}^g)| + |\text{ker}(\rho_r^g)\cap \text{ker}(\nu_{r-2}^g)|.
\end{equation}
\vspace{.2cm}

\noindent Using our formulae from Section \ref{sec:preliminaries}, this kernel is equal to $m_{r-2}^g + h_{r-1}^g$ if and only if (\ref{eq:kercap}) holds. In case $|\text{ker}(\rho_r^g)\cap \text{ker}(\nu_{r-2}^g)|=0$ does hold, the cokernel is given by

\vspace{.2cm}
\[
    |\text{coker}(\lambda_r^{1,g})| \; = \; 2 h_{r-3}^g + h_{r-4}^g - m_{r+1}^g.
\]
\vspace{.2cm}

\noindent Together with items 3 and 6 from Lemma \ref{lemma:lambdaker} we derive \ref{eq:rec3} for $r\geqslant 3g+5$ with equality holding, and by Poincar\'{e} duality, \ref{eq:rec1} for $r\leqslant 3g-2$ with equality. The case of $r=3g+2$ is similarly handled. From this argument it is clear that (\ref{eq:kercap}) holds if and only if \ref{eq:rec1}$-$\ref{eq:rec3} are equalities, and that more generally $|\text{ker}(\rho_r^g)\cap \text{ker}(\nu_{r-2}^g)|\geqslant 0$ implies the inequalities \ref{eq:rec1}$-$\ref{eq:rec3} for all $r$.
\end{proof}

\vspace{.4cm}

\noindent Recall from the introduction our claim that equality in \ref{eq:rec1}$-$\ref{eq:rec3} follows if $\nu_r^g$ has maximal rank for all $r$. We explain this here for \ref{eq:rec1}. For the range beyond the middle dimension, this asks for $\nu_r^g$ to be injective, and thus our claim from the introduction follows from Lemma \ref{lemma:kercap}. However, we can also see how surjectivity of $\nu_r^g$ in the range below the middle dimension would suffice: here the kernel of $\psi_r^{1,g}$ is computed to have dimension $h_{r-1}^g-m_{r-1}^g-m_{r-3}^g$. This is obtained by splitting off the kernel of $\nu_{r-1}^g$, which contributes $h_{r-1}^g-\check{n}_{r-1}^g$, then cancelling the remaining isomorphic part of $\nu_{r-1}^g$ against $\rho_{r-1}^g$, and accounting for the kernel of $\mu_{r-3}^g=\nu_{r-3}^g\oplus\rho_{r-3}^g$ left over. Computing the cokernels and applying (\ref{eq:lambda}) yields equality in \ref{eq:rec1}. 


\vspace{.7cm}

\vspace{.3cm}

\section{Computations for the genus 2 decomposition}\label{sec:genus2}

\begin{figure}[t]\label{fig:ss4}
  \centering{
\includegraphics[scale=1.1]{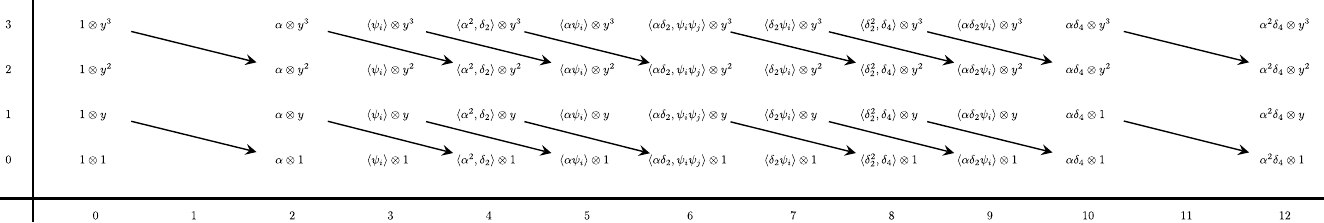}
}
  \caption{{\small{The $E_2$ page in the Leray-Serre spectral sequence for $N_3^\#$. }}}
\end{figure}

In this final section we sketch the computations that show $\nu_r^g$ is of maximal rank for all $r$ with $g=2$. None of these are needed for the results stated in the introduction.

One might try to prove equality in \ref{eq:rec1}$-$\ref{eq:rec3} by using other Mayer-Vietoris decompositions. For example, moving a level down from (\ref{eq:mv1}), we may consider the genus 2 decomposition

\vspace{.2cm}
\begin{equation}\label{eq:mv2}
	 N_{g+1}^\# \; = \;  N_2^+ \times N_{g-1}^\# \bigcup_{S^2\times N_{2}^\# \times N_{g-1}^\#} N_2^\# \times N_{g-1}^+
\end{equation}
\vspace{.2cm}

\noindent which may be described in a similar manner as was the genus 1 decomposition (\ref{eq:mv1}) in \cite[\S 4]{newstead-mv}. Just as in the previous case, we consider the Mayer-Vietoris sequence associated with (\ref{eq:mv2}). We have a map $\lambda_{r}^{2,g-1}$ which we decompose into two parts, as follows:

\vspace{.25cm}

\begin{center}
\begin{tikzcd}[column sep=large] 
H_r(N_2^+\times N_{g-1}^\#) &  H_r(S^2\times N_2^\#\times N_{g-1}^\#) \arrow[to=1-1, red]  \arrow[to=1-3, blue] & H_r(N_2^\#\times N_{g-1}^+) 
\end{tikzcd}
\end{center}

\vspace{.25cm}

\noindent We also have the analogue of (\ref{eq:lambda}) from the exactness of the Mayer-Vietoris sequence:

\vspace{.2cm}
\begin{equation}
	h_r^{g+1} \; = \; |\text{coker}(\lambda_r^{2,g-1})| + |\text{ker}(\lambda_{r-1}^{2,g-1})|\label{eq:lambda2}
\end{equation}
\vspace{.2cm}

\noindent As before, we expand $\lambda_r^{2,g-1}$ into its various K\"{u}nneth components, and obtain the diagram in Figure \ref{fig:lambda2}. Here we note that the betti numbers $\check{n}_r^{g-1}$ of $N_{g-1}^+$ are easily computed from our knowledge of $h_r^2$ from Section \ref{sec:genus1} and the equations in Section \ref{sec:preliminaries}. These are listed in Figure \ref{table:2}. All of the unboxed data in the table is computed from the formulae in Section \ref{sec:preliminaries}. We will momentarily sketch how one can fill in the boxed data. Here we remark that after computing this data and attempting to adapt the Mayer-Vietoris argument of Section \ref{sec:mainarg} to this situation, it becomes apparent that more information about the maps $\nu_r^g$ and their interactions with the $\rho_r^g$ is required in order to compute the relevant kernels and cokernels.

We can compute the data in Figure \ref{table:2} by specializing to the $2+1$ and $2+2$ Mayer-Vietoris decompositions, setting $g=2$ and $g=3$ in (\ref{eq:mv2}). To carry this out we need the $\Z/2$ betti numbers of the moduli spaces $N_3^\#$ and $N_4^\#$. These are computed via Proposition \ref{prop:ss}, which uses the Leray-Serre spectral sequence, and the ring structures of $H^\ast(N_g;\F)$ for $g\in\{3,4\}$, which are available from \cite{ss-mod2}. See Figure \ref{fig:ss2} for an illustration of the genus 3 case. The numbers obtained are of course what appear in Figure \ref{fig:introtables}, and agree with the general conjectural recursions.

\begin{figure}[t]
  \centering{
\includegraphics[scale=1.05]{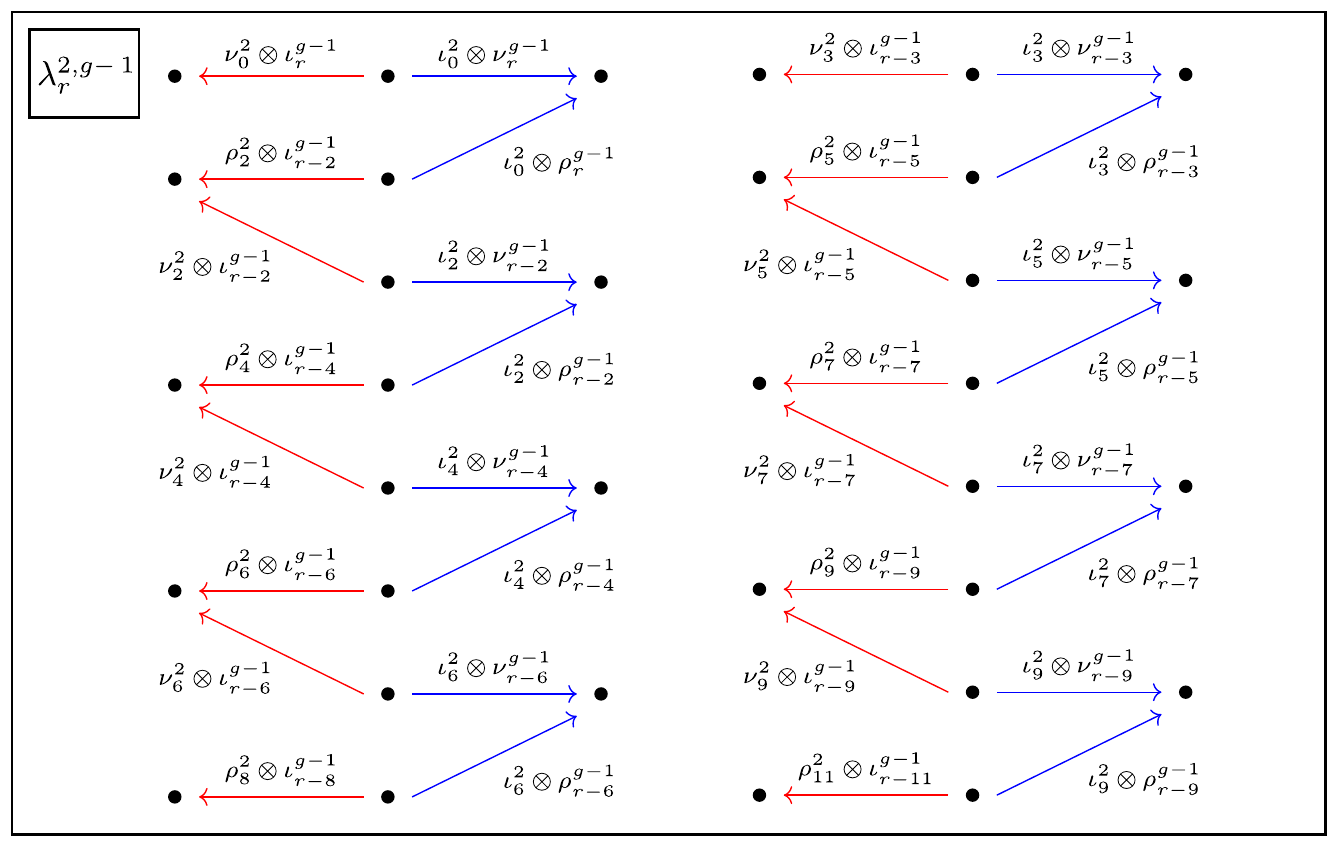}
}
\vspace{.2cm}
\caption{{\small{The map $\lambda_r^{2,g-1}$ expanded using the K\"{u}nneth Theorem.}}}\label{fig:lambda2}
\end{figure}

The $2+1$ Mayer-Vietoris decomposition of the genus 3 moduli space, which can also be viewed as one of the genus 1 decompositions (\ref{eq:mv1}), can be used to compute the following, in the listed order:

\vspace{.25cm}

\begin{enumerate}
    \item Use $\lambda_3^{1,2}$ to conclude that $\nu_2^2$ is an isomorphism.
    \item Use $\lambda_{11}^{1,2}$ to conclude that $\nu_9^2$ is an isomophism.
    \item Use $\lambda_6^{1,2}$ to conclude that $|\text{ker}(\nu_{3}^2)\cap\text{ker}(\mu_5^2)|=1$, implying that $\nu_5^2$ has rank 4 or 5.
    \item Use $\lambda_8^{1,2}$ to conclude that $|\text{ker}(\nu_6^2)\cap\text{ker}(\rho_8^1)|=0$, implying that $\nu_6^2$ has rank 4 or 5.
\end{enumerate}

\vspace{.25cm}

\noindent In each step we use the diagram of maps in Figure \ref{fig:lambda2} with $g=2$, the appropriate value of $r$, and linear algebra over $\F$ just as in Section \ref{sec:mainarg}. In particular, the key device is our use of the relation (\ref{eq:lambda2}) along with our aforementioned knowledge of the betti numbers $h_r^3$, which constrains the possible dimensions of the kernels and cokernels of $\lambda_{r}^{1,2}$. We mention that we can deduce a bit more than what is listed in item 3, from its computation: the kernel of the map $\nu_3^2 (\rho_5^2)^{-1}\nu_5^2$ is 1-dimensional. This information is useful for the reader who wishes to complete the subsequent steps.

\begin{figure}[t]
  \centering{
  \setlength{\extrarowheight}{7pt}
\begin{tabularx}{0.65\textwidth}{c |*{6}{Y}}
$r$        & $h_r^2$ & $\check{n}_r^2$ & $\mu_{r}^2$ & $ \rho_{r}^2$ & $\nu_{r}^2$ \tabularnewline \midrule
0           & $1$ & $1$ & $1_1^1$ & $0_0^1$ &  $1_1^1$\tabularnewline
1           & $0$ & $0$ & $0_0^0$ & $0_0^0$ &  $0_0^0$ \tabularnewline 
2           & $1$ & $1$ & $1_2^1$ & $1_1^1$ &   \fbox{$1_1^{1\phantom{1}}$} \tabularnewline
3           & $5$ & $4$ & $4_5^4$ & $0_0^4$ &   $4_5^4$\tabularnewline 
4           & $5$ & $1$ & $1_6^1$ & $1_1^1$ & \fbox{$\text{1}_5^{1\phantom{1}}$}\tabularnewline 
5           & $5$ & $5$ & $5_{10}^5$ & $5_5^5$ & \fbox{$\text{5}_5^{5\phantom{1}}$}\tabularnewline
6           & $5$ & $11$ & $5_{10}^{11}$ & $5_5^{11}$ & \fbox{$\text{5}_5^{11}$} \tabularnewline
7           & $1$ & $5$ & $5_6^5$ & $5_5^5$ &  \fbox{$\text{1}_1^{5\phantom{1}}$}\tabularnewline 
8           & $0$ & $1$ & $1_5^1$ & $1_5^1$ &   $0_0^1$\tabularnewline
9           & $1$ & $1$ & $1_2^1$ & $1_1^1$ &   \fbox{$1_1^{1\phantom{1}}$}\tabularnewline 
10           & $0$ & $0$ & $0_0^0$ & $0_0^0$ & $0_0^0$\tabularnewline 
11           & $0$ & $0$ & $0_1^0$ & $0_1^0$ & $0_0^0$\tabularnewline
\end{tabularx}
\vspace{.3cm}
}
  \caption{Genus 2 data.}\label{table:2}
  \vspace{.4cm}
\end{figure}

We may then proceed to use the $2+2$ decomposition of the genus 4 moduli space in a similar fashion to complete the following two steps:
 
\vspace{.25cm}

\begin{enumerate}
    \item[5.] Use $\lambda_7^{2,2}$ to conclude that $\nu_4^2$ is non-zero, and hence surjective.
    \item[6.] Use $\lambda_{13}^{2,2}$ to conclude that $\nu_7^2$ is non-zero, and hence injective.
\end{enumerate}

\vspace{.25cm}

\noindent It then remains to show that the ranks of $\nu_5^2$ and $\nu_6^2$ are $5$, instead of $4$. This computation is less direct. However, the joint constraints imposed by inspecting $\lambda_{r}^{2,2}$ for $r=8,9,10,12$ lead to the resolution of this claim, which, although entirely elementary, is somewhat tedious. For the reader interested in following this computation through we include the following table, which lists the final dimensions for some of the relevant kernels and cokernels.\\

\vspace{.5cm}

\begin{figure}[h]
\centering{
\setlength{\extrarowheight}{4pt}
{{\small{
\begin{tabularx}{1\textwidth}{c |*{22}{Y}}
$r$ & 0 & 1 & 2 & 3 & 4 & 5 & 6 & 7 & 8 & 9 & 10 & 11 & 12 & 13 & 14 & 15 & 16 &17 & 18 & 19 & 20 & 21 \tabularnewline \midrule
$h_r^4$ & 1 & 0 & 1 & 8 & 1 & 8 & 29 & 9 & 37 & 93 & 93 & 93 & 93 & 37 & 9 & 29 & 8 & 1 & 8 & 1 & 0 & 1 \tabularnewline
$|\text{cok}(\lambda_r^{2,2})|$ & 1 & 0 & 1 & 8 & 1 & 8 & 29 & 8 & 29 & 68 & 85 & 68 & 85 & 20 & 1 & 12 & 0 & 0 & 0 & 0 & 0 & 0 \tabularnewline
$|\text{ker}(\lambda_r^{2,2})|$ & 0 & 0 & 0 & 0 & 0 & 0 & 1 & 8 & 25 & 8 & 25 & 8 & 17 & 8 & 17 & 8 & 1 & 8 & 1 & 0 & 1 & 0 \tabularnewline
\end{tabularx}
}}}
}
\end{figure}

\vspace{.7cm}

\bibliography{main}
\bibliographystyle{alpha}

\vspace{.85cm}

\footnotesize

  \textsc{Simons Center for Geometry and Physics,
    Stony Brook, NY}\par\nopagebreak
  \textit{E-mail address:}\;\texttt{cscaduto@scgp.stonybrook.edu}

\vspace{.35cm}

    \textsc{Department of Mathematics, University of California,
    Los Angeles, CA}\par\nopagebreak
  \textit{E-mail address:}\;\texttt{mstoffregen@math.ucla.edu}

\end{document}